\documentclass[12pt]{amsart}
\usepackage[utf8]{inputenc}

\usepackage{amsrefs}
\usepackage{supertabular}
\usepackage{amsxtra,amssymb,amsmath,amscd,url,enumitem}

\setenumerate{itemsep=5pt}
\setitemize{itemsep=5pt}

\newlist{enumth}{enumerate}{1}
\setlist[enumth]{label=\emph{(\arabic*)}, ref=(\arabic*)}

\usepackage[utf8]{inputenc}
\usepackage[english]{babel}
\usepackage{fullpage}
\usepackage{scrtime}
\usepackage{datetime}
\usepackage[all]{xy}
\usepackage{tikz-cd}
\usepackage{mathrsfs}
\usepackage{xspace}

\usepackage{tensor}
\usepackage{braket}

\usepackage{todonotes}
\usepackage{tensor}

\usepackage[hypertexnames=false,
  pdfpagemode=UseNone,
  breaklinks=true,
  extension=pdf,
  colorlinks=true,
  linkcolor=blue,
  citecolor=red,
  urlcolor=blue,
 ]{hyperref}

\usepackage{needspace}

\shortdate

\DeclareMathSymbol{A}{\mathalpha}{operators}{`A}%
\DeclareMathSymbol{B}{\mathalpha}{operators}{`B}%
\DeclareMathSymbol{C}{\mathalpha}{operators}{`C}%
\DeclareMathSymbol{D}{\mathalpha}{operators}{`D}%
\DeclareMathSymbol{E}{\mathalpha}{operators}{`E}%
\DeclareMathSymbol{F}{\mathalpha}{operators}{`F}%
\DeclareMathSymbol{G}{\mathalpha}{operators}{`G}%
\DeclareMathSymbol{H}{\mathalpha}{operators}{`H}%
\DeclareMathSymbol{I}{\mathalpha}{operators}{`I}%
\DeclareMathSymbol{J}{\mathalpha}{operators}{`J}%
\DeclareMathSymbol{K}{\mathalpha}{operators}{`K}%
\DeclareMathSymbol{L}{\mathalpha}{operators}{`L}%
\DeclareMathSymbol{M}{\mathalpha}{operators}{`M}%
\DeclareMathSymbol{N}{\mathalpha}{operators}{`N}%
\DeclareMathSymbol{O}{\mathalpha}{operators}{`O}%
\DeclareMathSymbol{P}{\mathalpha}{operators}{`P}%
\DeclareMathSymbol{Q}{\mathalpha}{operators}{`Q}%
\DeclareMathSymbol{R}{\mathalpha}{operators}{`R}%
\DeclareMathSymbol{S}{\mathalpha}{operators}{`S}%
\DeclareMathSymbol{T}{\mathalpha}{operators}{`T}%
\DeclareMathSymbol{U}{\mathalpha}{operators}{`U}%
\DeclareMathSymbol{V}{\mathalpha}{operators}{`V}%
\DeclareMathSymbol{W}{\mathalpha}{operators}{`W}%
\DeclareMathSymbol{X}{\mathalpha}{operators}{`X}%
\DeclareMathSymbol{Y}{\mathalpha}{operators}{`Y}%
\DeclareMathSymbol{Z}{\mathalpha}{operators}{`Z}%

\renewcommand{\leq}{\leqslant}
\renewcommand{\geq}{\geqslant}
 
\numberwithin{equation}{section}

\newcommand{\uple}[1]{\text{\boldmath${#1}$}}

\renewcommand{\mathcal}{\mathscr}

\newcommand{\dunion}{\mathrel{\uple{\sqcup}}}
\newcommand{\ccirc}{\mathrel{\odot}} %circledcirc}

\newcommand{\Cc}{\mathbb{C}}
\newcommand{\Nn}{\mathbb{N}}

\newcommand{\Zz}{\mathbb{Z}}

\newcommand{\Rr}{\mathbb{R}}

\newcommand{\Qq}{\mathbb{Q}}

\newcommand{\mcN}{\mathscr{N}}

\newcommand{\expect}{\mathbb{E}}

\def\loccit{loc.\kern3pt cit.{}\xspace}
\def\cf{see\kern.3em}
\def\Cf{See\kern.3em}
\def\eg{e.g.\kern.3em}
\def\ie{i.e.,\ }
\def\resp{\text{resp.}\kern.3em}

\newcommand{\mods}[1]{\,(\mathrm{mod}\,{#1})}

\newcommand{\injecte}{\hookrightarrow}

\DeclareMathOperator{\Fix}{Fix}

\DeclareMathOperator{\Reel}{Re}

\DeclareMathOperator{\tr}{{}^t\!}

\DeclareMathOperator{\Tr}{Tr}
\DeclareMathOperator{\Hom}{Hom}
\DeclareMathOperator{\End}{End}
\DeclareMathOperator{\Aut}{Aut}

\DeclareMathOperator{\dual}{D}

\DeclareMathOperator{\Rep}{\underline{\mathrm{Rep}}}
\DeclareMathOperator{\Del}{\underline{\mathrm{Rep}}}

\renewcommand{\rho}{\varrho}

\DeclareMathOperator{\GL}{\mathbf{GL}}
\DeclareMathOperator{\Aff}{\mathbf{Aff}}

\DeclareMathSymbol{\gena}{\mathord}{letters}{"3C}
\DeclareMathSymbol{\genb}{\mathord}{letters}{"3E}

\theoremstyle{plain}
\newtheorem{theorem}{Theorem}[section]
\newtheorem*{theorem*}{Theorem}
\newtheorem{lemma}[theorem]{Lemma}
\newtheorem{corollary}[theorem]{Corollary}
\newtheorem{conjecture}[theorem]{Conjecture}

\newtheorem{proposition}[theorem]{Proposition}

\theoremstyle{remark}

\theoremstyle{definition}

\newtheorem{definition}[theorem]{Definition}

\newtheorem{example}[theorem]{Example}
\newtheorem{remark}[theorem]{Remark}

\newcommand{\mcC}{\mathscr{C}}

\renewcommand{\geq}{\geqslant}
\renewcommand{\leq}{\leqslant}

\DeclareMathOperator{\Std}{Std}

\begin{document}

\title{Fixed-point statistics from spectral measures on tensor envelope
  categories}

\author{Arthur Forey}
\address[A. Forey]{Univ.\,Lille, CNRS, UMR 8524 - Laboratoire Paul Painlevé, 59000 Lille, France} 
\email{arthur.forey@univ-lille.fr}

\author{Javier Fres\'an}
\address[J. Fres\'an]{Sorbonne Université and Université Paris Cité, CNRS, IMJ-PRG, 75005 Paris, France}
\email{javier.fresan@imj-prg.fr}

\author{Emmanuel Kowalski}
\address[E. Kowalski]{D-MATH, ETH Z\"urich, R\"amistrasse 101, 8092 Z\"urich, Switzerland} 
\email{kowalski@math.ethz.ch}



\begin{abstract}
  We prove old and new convergence statements for fixed-points
  statistics and characters of symmetric groups using tensor
  envelope categories, such as the Deligne\nobreakdash--Knop category of
  representations of the ``symmetric group'' $S_t$ for an
  indeterminate~$t$. We also speculate on a generalization of Chebotarev's density theorem to pseudopolynomials. 
\end{abstract}

\maketitle 

\section{Introduction}

Spectral measures associated to operators on Hilbert
spaces are key tools in functional analysis and its applications, for
instance to quantum mechanics and ergodic theory. Recall that a
continuous normal linear operator $u\colon E\to E$ on a Hilbert space $E$ has a compact spectrum and that, for each vector~$x\in E$, there exists a unique bounded positive Radon measure~$\mu_x$ on the set of complex numbers~$\Cc$ such that the equality 
\[
\int_{\Cc}f\, d\mu_x= \langle x\,|\,f(u)x\rangle
\]
holds for all continuous functions~$f \colon \Cc \to \Cc$; see, for instance, \cite[IV,\,p.\,190,\,déf.\,2]{ts345}. This
measure is supported on the spectrum of $u$ and is called the \emph{spectral measure of~$u$ relative to~$x$}. In
particular, we have
\[
\int_{\Cc}z^a\bar{z}^b\, d\mu_x(z)= \langle x\,|\,u^a(u^*)^b(x)\rangle
\]
for all non-negative integers~$a$ and~$b$. In this paper, we consider an
analogue of this last relation for objects in symmetric monoidal
categories.

\begin{definition}[Spectral measure]\label{def-spectral}
  Let~$\mcC$ be a symmetric monoidal category with an en\-dofunctor
  $\dual$, and let $i$ be a complex-valued invariant of~$\mcC$, by which
  we mean a map from the set of isomorphism classes of objects of~$\mcC$
  to~$\Cc$. Let~$M$ be an object of~$\mcC$. A positive measure~$\mu=\mu(i, M)$
  on~$\Cc$ is called a \emph{spectral measure of~$M$ relative to~$i$} if
  the equality
  \[
    \int_{\Cc} z^a\, \bar{z}^b\, d\mu(z)=i(M^{\otimes a}\otimes
    \dual(M)^{\otimes b})
  \]
  holds for all non-negative integers $a$ and~$b$.
\end{definition}

We will think of $\dual$ as a duality functor on $\mcC$, although no extra condition is required for this definition. In general, the measure $\mu$ depends on $i$ and $M$ and might not be unique (see Remark \ref{non-uniqueness}). The basic motivation is provided by the following example.

\begin{example}\label{ex:SatoTate}
  Let $r\geq 1$ be an integer, and let $G\subset \GL_r(\Cc)$ be a compact
  group with probability Haar measure~$\nu$. Let
  $\mcC$ be the category of finite-dimensional continuous complex representations of~$G$,
  and~$D$ the contragredient endofunctor of~$\mcC$. By representation
  theory of compact groups, the direct image of~$\nu$ by the trace $\Tr \colon G \to \Cc$ is a spectral measure $\mu=\Tr_\ast(\nu)$  of the ``tautological''
  object of~$\mcC$ corresponding to the inclusion of~$G$ in~$\GL_r(\Cc)$, relative to the invariant given on a representation $\rho$ by 
 \[
 i(\rho)=\dim_{\Cc}(\rho^G)=\dim_{\Cc} \Hom(1, \rho), 
 \] where $1$ denotes the trivial one-dimensional representation of $G$. 
  In number theory, measures of this kind are often called \emph{Sato--Tate
  measures}, the original example being $\mathrm{SU}_2 \subset \GL_2(\Cc)$. 
\end{example}

By a somewhat ad hoc construction, one can also see the spectral
measures from functional analysis as instances of spectral measures in
the sense of Definition \ref{def-spectral}.

\begin{example} Let $E$ be a Hilbert space, and let $A$ be a commutative unitary $C^\ast$-subalgebra of $\End(E)$, \eg the closure of the span of all $u^a$ and $(u^\ast)^b$ for some normal endomorphism $u$. We consider the category $\mcC$ with objects the elements of $A$ and morphisms given by 
\[
\Hom_{\mcC}(u, v)=\begin{cases} \Cc\cdot 1_u & \text{if $u=v$,} \\ \{0\} & \text{else.} \end{cases} 
\] We endow $\mcC$ with the tensor product given on objects by $u \otimes v=u \circ v$, and by the obvious rule on morphisms, \ie the tensor product of non-zero morphisms $f=s 1_{u_1}$ and~$g=t 1_{u_2}$ is equal to $f \otimes g=st 1_{u_1 \otimes u_2}$. The unit object is the identity map $1_E \in A$, and the category is symmetric monoidal because $A$ is commutative. We consider the ``duality'' functor $D(u)=u^\ast$. 

Let us now fix a vector $x \in E$ and consider the invariant $i$ on $\mcC$ defined by the assignment~$i_x(u)=\langle x\,|\,u(x)\rangle$ for all $u \in A$. Then the equality
\[
i_x(u^{\otimes a} \otimes D(u)^{\otimes b})=\langle x\,|\,u^a(u^*)^b(x)\rangle
\] holds for all non-negative integers $a$ and $b$, and hence the functional-analytic spectral measure~$\mu_x$ is a spectral measure of $u$ relative to the invariant $i_x$. 
\end{example}

Our first main result is a new proof of a statement that goes back to
the very early studies of probability theory through the analysis of
card games and the like (see the historical paper of
Takács~\cite{takacs} for references). Interestingly, the tensor
categories that will arise in the proof are the categories of
representations of the ``symmetric group'' $S_t$ for an indeterminate~$t$ of
Deligne \cite{deligne} and Knop~\cite{knop}. Another, rather different, construction of this category has recently been given by Harman and Snowden~\cite[\S\,15]{harman-snowden}.

\begin{theorem}[``Problème des rencontres''; Montmort~\cite{montmort};
  N. Bernoulli I; de Moivre~\cite{de-moivre}]\label{th-fix} Let~$(X_n)_{n\geq 1}$ be a sequence of random variables with $X_n$ a
  uniformly chosen random permutation in the symmetric group~$S_n$. The
  sequence $(|\Fix(X_n)|)_{n\geq 1}$, where~$\Fix(\sigma)$ denotes the
  set of fixed points of~$\sigma\in S_n$, converges in law to a Poisson
  distribution with parameter~$1$.
\end{theorem}

Recall that the Poisson distribution with parameter a positive real number~$\lambda$ is
the measure~$P_{\lambda}$ supported on non-negative integers given by 
\[
  P_{\lambda}(\{r\})=e^{-\lambda}\frac{\lambda^r}{r!}
\]
for all integers $r\geq 0$. The meaning of the statement (and how it was
originally proved) is therefore that the formula
\begin{equation}\label{eqn:limit}
\lim_{n\to +\infty} \frac{1}{n!}  |\{\sigma\in S_n \text{ with } 
|\Fix(\sigma)|=r\}|=\frac{1}{e}\frac{1}{r!}
\end{equation}
holds for all integers $r\geq 0$. Neither categories nor spectral
measures appear in the statement; the link comes from the fact that the
limit Poisson distribution arises as the spectral measure of a suitable object
in the Deligne--Knop category $\mathcal{C}_t=\Rep(S_t)$. This is a $\Cc(t)$\nobreakdash-linear semisimple tensor category that ``interpolates'' the categories of representations of the symmetric groups $S_n$. From that point of view, an interesting feature of our proof is that it shows how the Poisson distribution (maybe the most natural measure on non-negative integers) is some kind of analogue of the Sato--Tate measures from Example \ref{ex:SatoTate}. In Section~\ref{sec-3.3}, we will see a similar
statement for the complex gaussian distribution. By Chebotarev's density theorem, the probability that a random permutation in $S_n$ has $r$ fixed points governs the asymptotic density of the set of primes $p$ such that a generic polynomial of degree $n$ with integer coefficients has $r$ roots modulo $p$. In
Section~\ref{sec-speculations}, we will speculate on a generalization of Chebotarev's density theorem that could explain the occurrence of the limit \eqref{eqn:limit} in numerical experiments involving certain pseudopolynomials. 

\needspace{10\baselineskip}
Our second main result is the following new theorem.

\begin{theorem}\label{th-fi}
  Let~$m\geq 1$ be an integer and let~$\lambda$ be a partition
  of~$m$ with parts
\[
  \lambda_1\geq \lambda_2\geq \cdots.
\] For~$n\geq m+\lambda_1$, let~$\pi_{\lambda,n}$ be the representation
  of~$S_n$ corresponding to the partition
 \[
  \lambda^{(n)}=(n-m,\lambda_1,\lambda_2,\ldots),
 \]
  and let~$\chi_{\lambda,n}\colon S_n\to\Cc$ be its character. Then the sequence of measures~$(\chi_{\lambda,n}(X_n))_{n \geq m+\lambda_1}$, where as
  before~$X_n$ is a uniformly distributed random permutation on~$S_n$, converges in law as $n\to+\infty$ to a spectral measure of the simple
  object~$x_{\lambda,t}$ of the Deligne--Knop category~$\mathcal{C}_t$ associated to the partition~$\lambda$,
  relative to the invariant
  \[
  i(M)=\dim_{\Cc(t)} \Hom(1_{t}, M),
  \] where $1_t$ is the unit objet of~$\mathcal{C}_t$.  
\end{theorem}

This second result has a corollary which relates it to the theory of
FI-modules of Church, Ellenberg and Farb~\cite{cef}. An \emph{FI-module} over a field $k$ is a functor $V$ from the category with objects finite sets and morphisms injective maps to the category of $k$-vector spaces. For each~$n \geq 0$, we write $V_n$ for the image of the set $\{1, \dots, n\}$; note that $V_n$ has a natural structure of representation of $S_n$. An FI-module $V$ is called \emph{finitely generated} if there exists a finite set of elements $x_i \in V_{n_i}$ that do not ``lie'' on a proper subfunctor of $V$. We refer the reader to the introduction of~\cite{cef} for a list of examples of finitely-generated FI-modules arising in algebra, geometry and
  topology.

\begin{corollary}
  Let~$V=(V_n)_{n\geq 0}$ be a finitely-generated \emph{FI}-module
  over~$\Cc$. For $n\geq 0$, let~$\xi_n$ be the character of~$V_n$ as an
  $S_n$-representation. The sequence of measures $(\xi_n(X_n))_{n \geq 0}$ converges in law as $n\to+\infty$ to a combination of spectral measures of objects of~$\mcC_t$
  relative to the invariant~$i$ from Theorem~\ref{th-fi}.
\end{corollary}

\begin{proof}
  A fundamental result of the theory of
  FI-modules~\cite[Prop.\,3.3.3]{cef} implies the existence of a 
  polynomial $Q\in \Cc[(T_{\lambda})_{\lambda}]$ (in indeterminates
  parameterized by all partitions of all integers~$m\geq 1$) satisfying $\xi_n=Q((\chi_{\lambda,n})_{\lambda})$
  for all large enough $n$, and hence the corollary follows
  immediately from Theorem~\ref{th-fi}.
\end{proof}

%%%%%%%%%%%

%excerpted from a longer work in progress~\cite{ffk}, whose status and evolution are however unpredictable. 

This paper is intended as a first glimpse on a subject that deserves further exploration. We hope that the simple example of an application of
spectral measures to classical problems will motivate the reader's
interest in this notion.

\subsection*{Conventions.} Let $X$ be a set. A \emph{partition} of~$X$ is a
set of non-empty subsets of~$X$, pairwise disjoint and with union equal
to~$X$. Note that this definition constrasts with that of Bourbaki (E,
II, p.~29, déf.~7), where a partition is a \emph{family} of subsets
of~$X$, allowing the empty set.  
  
\subsection*{Acknowledgements.} We would like to thank Jordan Ellenberg and Johannes Flake for fruitful discussions on spectral measures, as well as the anonymous referee for his or her interest in this notion and the many kind and generous suggestions to develop it further.   During the preparation of this project, A.\,F. was partially supported by the SNF Ambizione grant PZ00P2\_193354 and J.\,F. was partially supported by the grant ANR-18-CE40-0017 of the Agence Nationale de la Recherche.
  
\section{Existence and uniqueness of spectral measures in tensor categories}
  
From now on, we only consider $k$-linear tensor categories, for some
field~$k$, in the sense of Deligne~\cite[1.2]{deligne2}. We always use the
duality functor of such a category as the endofunctor~$D$ in
Definition~\ref{def-spectral}. An object $M$ is called \emph{self-dual}
if there exists an isomorphism $M\simeq \dual(M)$.
% \emph{tensor category}, we mean a rigid monoidal abelian $k$-linear
% category, for some field~$k$, such that the endomorphism ring of the
% unit object is isomorphic to~$k$ (see~\cite[1.2]{deligne2}, and compare
% with~\cite[Def.\,4.1.1]{etingof-al}, which assumes that $k$ is
% algebraically closed).

\begin{proposition}[Spectral measures for self-dual objects]
  Let~$\mcC$ be a tensor category in which every object is
  self-dual, and let $i$ be an $\Rr$-valued \emph{additive} invariant of~$\mcC$.
  If 
  \begin{equation}\label{eq-pos2}
    2i(M\otimes N)\leq 
    i(M\otimes M)+i(N\otimes N)
  \end{equation}
  holds for all objects~$M$ and~$N$ of $\mcC$, then every object
  of~$\mcC$ admits a spectral measure relative to~$i$ which is supported
  on~$\Rr$.
\end{proposition}

\begin{proof}
  By the solution of the Hamburger moment problem (see,
  e.g.,~\cite[Th.\,3.8]{schmudgen}), a sequence~$(\mu_a)_{a\geq 0}$ of
  real numbers is the sequence of moments of a positive Borel measure
  $\mu$ on~$\Rr$ if and only if the inequality
  \begin{equation}\label{eq-ham}
    \sum_{1\leq a,b\leq A} \alpha_a \alpha_b\mu_{a+b}\geq 0
  \end{equation}
  holds for all integers $A\geq 1$ and all real
  numbers~$\alpha_a$. Therefore,~$M$ admits a real spectral measure relative
  to~$i$ if and only if the values $\mu_a=i(M^{\otimes a})$ satisfy this
  condition.

  We first consider the case where $\alpha_a$ are integers. Setting 
 \[
  P=\bigoplus_{\alpha_a\geq 0}\alpha_aM^{\otimes a} \quad\quad \text{and} \quad\quad
  N=\bigoplus_{\alpha_a\leq -1}(-\alpha_a)M^{\otimes a}, 
 \] we then get 
 \[
   \sum_{1\leq a,b\leq A} \alpha_a \alpha_b\mu_{a+b} =
   i(P\otimes P)+i(N\otimes N)
   -2i(P\otimes N), 
 \]
 and hence the assumption~(\ref{eq-pos2}) implies the inequality 
 \[
 \sum_{1\leq a,b\leq A} \alpha_a \alpha_b\mu_{a+b}\geq 0. 
\]
This extends to~$\Qq$ by homogeneity, and to~$\Rr$ by
continuity. 
\end{proof}

If not all objects are self-dual (as it often happens), then the situation is
more subtle, because the moment problem on~$\Cc$ is more challenging
than that on~$\Rr$: the analogue of the positivity condition above is not
sufficient to ensure the existence of a positive measure on~$\Cc$ with
given moments. However, under an extra growth condition, one obtains
both existence and uniqueness of the spectral measure. 

\begin{proposition}[Spectral measures for general objects]\label{prop:existencegeneralobjects}
  Let~$\mcC$ be a tensor category.  Let~$i$ be a $\Cc$-valued additive
  invariant of~$\mcC$. Suppose that the inequality
  \begin{equation}\label{eq-pos3}
    i(M\otimes \dual(N))+i(\dual(M)\otimes N)\leq 
    i(M\otimes \dual(M))+i(N\otimes \dual(N))
  \end{equation}
  holds for all objects~$M$ and~$N$ of~$\mcC$.  Let~$M$ be an object
  of~$\mcC$ satisfying the \emph{Carleman condition}
  \begin{equation}\label{eq-carleman}
    \sum_{a\geq 1}i((M\otimes \dual(M))^{\otimes a})^{-1/(2a)}=+\infty.
  \end{equation}Then there exists a unique spectral measure for~$M$ relative to~$i$.
\end{proposition}

\begin{proof}
  This follows as above (\emph{mutatis mutandis} using complexification)
  from the fact (due to Nussbaum) that the Carleman 
  condition~(\ref{eq-carleman}) combined with the analogue of~(\ref{eq-ham}), is a
  sufficient condition for the existence \emph{and} uniqueness of a
  measure on~$\Cc$ with given moments; see for
  instance~\cite[Th.\,15.11]{schmudgen}.
\end{proof}

\begin{remark}
  The Carleman condition holds in particular if there exist~$c\geq 0$
  and~$r\geq 0$ such that the inequality $i((M\otimes\dual(M))^{\otimes n})\leq cr^n$
  holds for all non-negative integers~$n$. This is a frequent occurrence, but
  it corresponds to measures with compact support (compare with
  Deligne's ``subexponential growth theorem''; see~\cite[Th.\,9.11.4]{etingof-al}).
\end{remark}

\begin{definition}[Positive invariants]
  An additive invariant~$i$ on~$\mcC$ is called a \emph{positive invariant} if
  it satisfies~(\ref{eq-pos3}) for all objects~$M$ and~$N$ of~$\mcC$. %If, in addition, all objects of~$\mcC$ admit a spectral measure, then we say that~$i$ is a \emph{complete positive invariant}.
\end{definition}

The following result gives a usable criterion to check that certain
invariants are positive.

\begin{proposition}\label{pr-general-decomp}
  Let~$\mcC$ be an essentially small tensor category. Let
  $\widehat{\mcC}$ be a set of objects of $\mcC$ such that every object
  of~$\mcC$ is isomorphic to a finite direct sum of objects from
  $\widehat{\mcC}$. We write~$\Zz^{(\widehat{\mcC})}$ for the set of functions $n \colon \widehat{\mcC}\to\Zz$ with finite support and $n_V=n(V)$ for all objects $V$ of $\widehat{\mcC}$. Let~$i$ be an additive invariant of~$\mcC$.  Then
  $i$ is positive if the bilinear form
  \[
  b(n,m)=\sum_{V,W\in\widehat{\mcC}}n_Vm_W\ i(V\otimes\dual(W))
  \]
  on~$\Zz^{(\widehat{\mcC})}$ is positive, i.e.,
  $ b(n,n)\geq 0$ for all functions $n\colon\widehat{\mcC}\to\Zz$ with finite
  support.
\end{proposition}

\begin{proof}
  Let~$M$ and~$N$ be objects of~$\mcC$, and represent them as direct sums
  \[
  M=\bigoplus_{V\in\widehat{\mcC}} m_{V}V,\quad\quad
  N=\bigoplus_{W\in\widehat{\mcC}} n_{W}W
  \]
  with only finitely many non-zero integers $m_V$, $n_W$.  By
  additivity, we obtain the formulas
  \begin{gather*}
    i(M\otimes\dual(N))+i(\dual(M)\otimes N)=
    \sum_{V,W} m_Vn_W\ i(V\otimes\dual(W))+
    \sum_{V,W} m_Vn_W\ i(\dual(V)\otimes W),
    \\
    i(M\otimes\dual(M))+i(N\otimes \dual(N))=
    \sum_{V,W} m_Vm_W\ i(V\otimes\dual(W))+
    \sum_{V,W} n_Vn_W\ i(V\otimes\dual(W))
  \end{gather*}
  so that we get
  \[
  \Bigl(i(M\otimes\dual(M))+i(N\otimes \dual(N)) \Bigr) -\Bigl(
  i(M\otimes\dual(N))+i(\dual(M)\otimes N)\Bigr)= b(m-n,m-n),
  \]
  and the result then follows from Proposition~\ref{prop:existencegeneralobjects}. 
\end{proof}

As a special case, we deduce:

\begin{corollary}\label{cor-general-invariant}
  Let~$k$ be a field.  Let $\mcC$ be any essentially small $k$-linear
  semisimple tensor category with unit object~$1_{\mcC}$ in which the $\Hom$ spaces are finite-dimensional. The formula
  \[
  i(M)=\dim_k \Hom(1_{\mcC},M)
  \] defines a positive invariant on~$\mcC$.
\end{corollary}

\begin{proof}
  We apply Proposition~\ref{pr-general-decomp} to the set $\widehat{\mcC}$ of
  isomorphism classes of simple objects of~$\mcC$. Then
  $i(V\otimes \dual(W))=0$ for $V$ and~$W$ in~$\widehat{\mcC}$, unless
  $V$ is equal to~$W$, so that the bilinear form~$b$ in the statement is
  diagonal in the canonical basis of $\Zz^{(\widehat{\mcC})}$, with
  diagonal coefficients equal to 
  $i(V\otimes \dual(V))=\dim_k \Hom(1_{\mcC},V\otimes\dual(V)) \geq 0$.
\end{proof}

\begin{remark}
  (1) The remainder of this paper will concentrate on the invariant of
  Corollary~\ref{cor-general-invariant}. However, there are other
  natural potential invariants that may be considered. One which seems
  quite interesting is the \emph{length} of an object in a semisimple category
  (where all objects have finite length). In the simplest case of the
  category of finite-dimensional complex representations of a finite
  group~$G$, it is an elementary exercise that the length is a positive invariant
  (and every object has a unique spectral measure relative to the
  length) if and only if the sum (without multiplicity) of the
  irreducible characters of~$G$ is non-negative. This holds for instance
  for all symmetric groups, but not all alternating groups (the latter
  experimentally). We hope to come back to this example in a later
  paper.
  
  (2) We emphasize that it is essential to impose the positivity of the spectral measure in Definition~\ref{def-spectral}: it is known by independent work of Boas and Pólya (see, e.g.,~\cite{boas}) that \emph{any} sequence of
  complex numbers is the sequence of moments of infinitely many complex
  measures on~$\Rr$.

  (3) In general, spectral measures are not uniquely determined given
  the object of interest, and only their moments are unambiguously known
  (see the conclusion of Remark~\ref{non-uniqueness} for a simple example
  where the spectral measure is not unique). %One can hope that (at least for certain categories or invariants) some additional condition or property might ``fix'' the spectral measure.
\end{remark}

We conclude this section with a simple observation.

\begin{proposition}\label{pr-image-measures}
  Let~$k$ be a field.  Let $\mcC$ be a $k$-linear tensor category with
  unit object~$1_{\mcC}$ and let~$i$ be a positive invariant on~$\mcC$.
  Let~$M$ be an object of~$\mcC$. Let~$\mu$ be a spectral measure
  for~$M$ relative to~$i$.

  \begin{enumth}
  \item For any non-negative integers $m$ and~$n$, the image measure
    $(z\mapsto z^m\bar{z}^n)_*\mu$ is a spectral measure
    for~$M^{\otimes a}\otimes \dual(M)^{\otimes b}$ relative to~$i$.
  \item The image measure $(z\mapsto 2\Reel(z))_*\mu$ is a spectral
    measure for~$M\oplus \dual(M)$ relative to~$i$.
  \end{enumth}
\end{proposition}

\begin{proof}
  We prove the second statement, the first being similar. The object
  $N=M\oplus \dual(M)$ is self-dual, so it suffices to consider
  $i(N^{\otimes a})$ for all integers~$a\geq 0$. From the isomorphism
  \[
  N^{\otimes a}\simeq \bigoplus_{0\leq b\leq a}
  \binom{a}{b} M^{\otimes b}\otimes \dual(M)^{\otimes (a-b)},
  \]
  and the definition of spectral measures, we get the equality
  \[
  i( N^{\otimes a})=\sum_{0\leq b\leq a} \binom{a}{b}
  \int_{\Cc}z^b\bar{z}^{b-a}d\mu(z)=\int_{\Cc} (z+\bar{z})^ad\mu(z),
  \] which means that $(z\mapsto 2\Reel(z))_*\mu$ is a spectral
    measure for~$N$. 
\end{proof}

\begin{remark}
  With obvious conventions, this proposition can be phrased and
  generalized as follows: for any polynomial
  $Q \in\Nn[z,\bar{z}]$, the measure $Q_*\mu$ is a spectral measure for
  the object $Q(M,\dual(M))$.

  More generally, one can raise the following natural question, for
  which we do not have good answers at the moment: given an object~$M$
  and some spectral measure $\mu(M)$ for~$M$ relative to some
  invariant~$i$, is there a ``natural'' definition of spectral measures
  $\mu(N)$, for all objects~$N$ of the tensor category generated by~$M$,
  such that $\mu(N)$ coincides with the measure~$Q_*\mu_M$ when $N=Q(M,D(M))$ as
  above?  Already when considering simple examples of Schur functors,
  such as symmetric powers (when they are defined), the answer is not
  clear. 
\end{remark}

\section{Tensor envelopes and fixed-point statistics}

Let $k$ be a field of characteristic zero and $t \in k$ an element. Deligne~\cite[Th.\,2.18]{deligne} defined a rigid $k$\nobreakdash-linear pseudo-abelian symmetric monoidal category $\Rep(S_t, k)$ by generators and relations, relying on some stability properties of representations of symmetric groups. If~$t$ is not a non\nobreakdash-negative integer $n \geq 0$, then $\Rep(S_t, k)$ is abelian and semisimple. If $t=n$, then the semisimplication of $\Rep(S_t, k)$ is equivalent to the category $\mathrm{Rep}(S_n)$ of $k$-linear representations of the symmetric group~$S_n$. We will mainly deal with the case where $k=\Cc(t)$ and $t$ is the indeterminate of $k$, which we simply denote by $\Rep(S_t)$. 

Knop~\cite{knop} discovered an alternative
approach to constructing new rigid symmetric monoidal categories which
is \emph{a priori} independent of ideas of interpolating other
categories; this leads to many more examples, and happens to recover in
a special case the categories of Deligne. The input data in Knop's construction is a base category $\mathcal{A}$ satisfying some regularity conditions, a field~$k$ and a degree function
$\delta$ which associates to every surjective morphism~$e$
in~$\mathcal{A}$ an element~$\delta(e)$ of~$k$, again subject to some
conditions. The resulting category is denoted
$\mathcal{T}(\mathcal{A},\delta)$ by Knop, and is called the
\emph{tensor envelope} of~$\mathcal{A}$ with respect to~$\delta$. For
the moment, it is sufficient for us to recall that every object~$x$
of~$\mathcal{A}$ defines an object $[x]$
of~$\mathcal{T}(\mathcal{A},\delta)$, which is always self-dual, and
that the $k$-linear space of morphisms from~$[x]$ to~$[y]$ admits as a
basis the set of all \emph{relations} from~$x$ to~$y$, i.e., the set of
all subobjects of the product~$x\times y$. To give some context, we spell out in Appendix~\ref{appendix} the construction of~$\mathcal{T}(\mathcal{A},\delta)$ in the special case relevant to
Theorem~\ref{th-fix}, namely when~$\mathcal{A}$ is the \emph{opposite} 
of the category of finite sets.

\subsection{Proof of Theorem~\ref{th-fix}} Let $P_1$ denote the Poisson distribution with parameter $1$. By the so-called \emph{Dobiński's formula} (see, e.g.,~\cite{pitman}), for each integer $k \geq 0$, the $k$-th moment
\[
\expect(P_1^k)=\frac{1}{e}\sum_{r=0}^\infty \frac{r^k}{r!}
\] agrees with the $k$-th Bell number, \ie the number of partitions of a set with $k$ elements (indeed, both sequences satisfy $a_0=0$ and the recurrence relation $a_{k+1}=\sum_{r=0}^k {k \choose r} a_r$). In particular, $\expect(P_1^k) \leq k^k$, so that the Carleman condition holds and $P_1$ is determined by its moments (as is well-known). Thanks to the method of moments (see, e.g.,~\cite[Th.\,30.2]{billingsley-basic}), to prove the convergence in law $|\Fix(X_n)| \to P_1$ as $n \to +\infty$, it suffices to prove that, for each integer~$k\geq 0$, the sequence of moments $(\expect(|\Fix(X_n)|^k))_{n \geq 1}$ converges to $\expect(P_1^k)$. 
  
 We first observe the equality~$|\Fix(X_n)|=\chi_n(X_n)$, where
  $\chi_n$ is the character of the ``standard'' permutation
  representation~$\Std_n$ of~$S_n$ acting on~$\Cc^n$. By basic representation
  theory of finite groups, we then get the expression 
  \begin{equation}\label{eqn:expectfix}
  \expect(|\Fix(X_n)|^k)= \frac{1}{n!}\sum_{\sigma \in
    S_n}\chi_n(\sigma)^k =\dim_{\Cc}
  \Hom_{\mathrm{Rep}(S_n)}(1_n,\Std_n^{\otimes k}),
  \end{equation}
  where~$1_n$ is the trivial one-dimensional representation of~$S_n$.
  
  We now appeal to the Deligne--Knop category~$\mcC_t=\Rep(S_t)$, first in the situation where~$t$ is the indeterminate in
  the field~$\Cc(t)$. Before pursuing the proof, we summarize the properties of $\mcC_t$ that will be useful for us:
  \begin{itemize}
  \item[(a)] Each finite set $X$ defines a self-dual object $[X]$ of $\mcC_t$, and these objects satisfy 
  \[
  \Hom_{\mcC_t}([X], [Y])=\Cc(t)\langle \text{partitions of $X \dunion Y$}\rangle. 
  \]
   \item[(b)] The tensor product of $[X]$ and $[Y]$ is the object $[X] \otimes [Y]=[X \dunion Y]$. 
   \item[(c)] The category $\mcC_t$ is a semisimple $\Cc(t)$-linear tensor category.
  \end{itemize} In particular, $\mcC_t$ contains objects $1_t=[\emptyset]$ and $\Std_t=[\{1\}]$, the first being the unit object for the tensor product. By Corollary~\ref{cor-general-invariant}, the assignment  
 \[
 i(M)=\dim_{\Cc(t)} \Hom_{\mcC_t}(1_t,M)
 \] defines a positive invariant on~$\mcC_t$.  
  
  \begin{lemma}\label{lm-standard} The object $\Std_t$ admits a unique spectral measure with respect to $i$, and this measure is equal to the Poisson distribution $P_1$. In particular,
  \begin{equation}\label{expectP1}
  \expect(P_1^k)=\dim_{\Cc(t)} \Hom_{\mcC_t}(1_t, \Std_t^{\otimes k}). 
  \end{equation}
  \end{lemma}
  
  \begin{proof} Recall that the object $\Std_t$ is self-dual. For each integer $k \geq 0$, the $k$-th tensor product $\Std_t^{\otimes k}$ is the object $[\{1, \ldots, k\}]$ of $\mcC_t$. Hence, $i(\Std_t^{\otimes k})=\dim_{\Cc(t)} \Hom_{\mcC_t}(1_t, \Std_t^{\otimes k})$ is the number of partitions of the set $\{1, \ldots, k\}$, which is also the $k$-th moment of $P_1$. 
  \end{proof}
  
  Combining \eqref{eqn:expectfix} and \eqref{expectP1}, the proof of Theorem~\ref{th-fix} then reduces to showing the equality 
  \[
  \lim_{n \to +\infty} \dim_{\Cc} \Hom_{\mathrm{Rep}(S_n)}(1_n,\Std_n^{\otimes k})=\dim_{\Cc(t)} \Hom_{\Rep(S_t)}(1_t, \Std_t^{\otimes k}). 
  \] For this, we use the variant $\mathcal{C}_z=\Rep(S_z, \Cc)$ of the Deligne--Knop category obtained by ``specializing'' the indeterminate $t$ to some fixed complex number $z$. Properties (a) and (b) from above still hold, and in particular 
  \[
  \dim_{\Cc(t)} \Hom_{\Rep(S_t)}(1_t, \Std_t^{\otimes k})=\dim_{\Cc} \Hom_{\mathcal{C}_z}(1_z, \Std_z^{\otimes k})
  \] since a basis of both vector spaces is given by the partitions of $\{1, \dots, k\}$. Unless $z$ is an integer $n \geq 0$, the category $\mathcal{C}_z$ is still semisimple. 
  
For integer values $z=n$, the semisimplication of $\mathcal{C}_n$ is equivalent, as a tensor category, to the category of finite-dimensional complex representations of $S_n$, an equivalence being given by a functor that maps an object of the form $[X]$ to the permutation representation on the space $V_X$ of functions $X \to \Cc^n$ (see~\cite[Th.\,9.8,\,Example\,1,\,p.\,606]{knop}). In particular, such a functor sends the object $[\emptyset]$ to the trivial one-dimensional representation $1_n$, and the object $[\{1\}]$ to the standard permutation representation $\Std_n$ on~$\Cc^n$. The semisimplification of $\mathcal{C}_n$ is the quotient category $\overline{\mathcal{C}}_n=\mathcal{C}_n/\mathcal{N}_n$, where $\mathcal{N}_n$ denotes the tensor radical of $\mathcal{C}_n$ (see~\cite[\S\,4.1]{knop}). It is a semisimple abelian tensor category by~\cite[Th.\,6.1]{knop}.
   
  \begin{lemma}\label{lm-dk-specialization}
  Let $n \geq 1$ be an integer. For each integer $k\geq 0$, the inequality 
  \[
  \dim_{\Cc} \Hom_{\mathrm{Rep}(S_n)}(1_n,\Std_n^{\otimes k})\leq
  \dim_{\Cc} \Hom_{\mathcal{C}_n}(1_t,\Std_n^{\otimes k})
  \]
  holds, with equality if and only if $n\geq k$.
\end{lemma}

\begin{proof} Let $X$ and $Y$ be finite sets. By definition of the quotient category, the objects of $\overline{\mathcal{C}}_n$ are the same as those of $\mathcal{C}_n$, and the morphisms between the representations $V_X$ and $V_Y$ of~$S_n$ corresponding to $[X]$ and $[Y]$ via the equivalence of categories  are given by 
\[
 \Hom_{\mathrm{Rep}(S_n)}(V_X, V_Y)= \Hom_{\overline{\mcC}_n}([X],[Y])=
  \Hom_{\mcC_n}([X],[Y])/\mcN_n([X],[Y]).
\] Therefore, we obtain an inequality 
\[
\dim_{\Cc} \Hom_{\mathrm{Rep}(S_n)}(V_X, V_Y) \leq \dim_{\Cc} \Hom_{\mathcal{C}_n}([X],[Y]), 
\] with equality if and only if $\mcN_n([X],[Y])$ is reduced to the zero morphism. Taking $X=\emptyset$ and~$Y=\{1, \ldots, k\}$, so that $V_X=1_n$ and $V_Y=\Std_n^{\otimes k}$, proves the first part of the statement. 

It remains to see when $\mcN_n(1_n, \Std_n^{\otimes k})$ is zero.  By a
  result of Knop~\cite[Cor.\,8.5]{knop}, this holds if and only if certain
  invariants $\omega_e$ in~$\Cc$ are non-zero for all indecomposable
  surjective morphisms $e \colon u\to v$ in the category
  $\mathsf{Set}^{opp}$ such that $u$ is a subquotient of
  $1_t\otimes \Std_t^{\otimes k}=\Std_t^{\otimes
    k}$. By~\cite[Ex.\,1,\,p.\,596]{knop}, this invariant is equal to $\omega_e=n-|v|$ for such morphisms; since indecomposable surjective morphisms
  $u\to v$ in $\mathsf{Set}^{opp}$ are injective maps of sets
  $v\injecte u$ satisfying $|v|=|u|-1$, and $u$ is a subquotient of
  $\Std_n^{\otimes k}=[\{1,\ldots, k\}]$, we have $|v|\leq k-1$, and hence
  $\omega_e=n-|v|\geq 1$ is non-zero for all $n \geq k$.
\end{proof}

This concludes the proof of Theorem~\ref{th-fix}. 

\begin{remark}\label{remarkotherasympt}
  (1) In comparison with other proofs, this abstract argument has the
  advantage of explaining, to some extent, where the Poisson
  distribution comes from.
  
  (2) It is natural to ask if similar ideas can be used to reprove other statements in the theory of random permutations, such as the fact that the sequence $(\ell_i(X_n))_{n \geq 1}$, where $\ell_i(\sigma)$ denotes the number of $i$-cycles in the decomposition of a permutation $\sigma$, converges in law to the Poisson distribution $P_{1/i}$. More ambitiously, one can try to count the number of cycles in a random permutation. 
  
  (3) To the best of our knowledge, the fact that the first moments of
  $|\Fix(X_n)|$ coincide with those of the Poisson distribution first appears in the work of Diaconis--Shashahani~\cite[Th.\,7]{d-s}.
\end{remark}  

\subsection{Fixed-point statistics for vector and affine spaces over finite fields} Knop's approach yields many more instances of tensor categories, and
the principles above are then applicable. As an example, we recover a
result of Fulman (proved in his $1997$ unpublished thesis) which appears
in a paper of Fulman and Stanton~\cite[Th.\,4.1]{fulman-stanton}.

\begin{proposition}[Fulman]\label{pr-fulman}
  Let~$E$ be a finite field and let $(X_n)_{n \geq 1}$ be a sequence of random
  variables with $X_n$ uniformly distributed in $\GL_n(E)$. The sequence
  $(|\Fix(X_n)|)_{n\geq 1}$, where~$\Fix(g)$ is the $1$-eigenspace
  of~$g\in\GL_n(E)$, converges in law as $n\to +\infty$. For $k\geq 0$,
  the $k$-th moment of the limiting distribution is equal to the number
  of vector subspaces of $E^k$. Moreover, the $k$-th moment of $|\Fix(X_n)|$ is equal to the limiting
  moment for $n\geq k$.
\end{proposition}

\begin{proof}
  We argue as in the proof of Theorem~\ref{th-fix}, using instead the base
  category $\mathsf{Vec}(E)$ of finite\nobreakdash-dimensional $E$-vector spaces and the degree function~$\delta(e\colon U\to V)=t^{\dim_E(\ker(e))}$
  for a surjective $E$-linear map to construct Knop's category $\mcC_t$.  We use as before the unit object
  $1_t=[\{0\}]$ and the standard object $\Std_t=[E]$, which is
  self-dual.

  Specializing to $t=|E|^n$ for some integer $n\geq 1$, the quotient
  $\overline{\mcC}_{|E|^n}$ is naturally equivalent to the category of
  finite-dimensional complex representations of~$\GL_n(E)$
  (see~\cite[Example\,5,\,p.\,606]{knop}). We obtain
  \[
  \dim_{\Cc} \Hom_{\GL_n(E)}(1_n,\Std_n^{\otimes k})\leq
  \dim_{\Cc(t)} \Hom_{\mcC_t}(1_t,\Std_t^{\otimes k}), 
  \]
  where $\Std_{n}$ is the $|E|^n$-dimensional permutation representation
  of $\GL_n(E)$ associated to its natural action on~$E^n$. As before,
  there is equality if the numerical invariants $\omega_e$ are non-zero
  for indecomposable surjective $E$-linear maps $e\colon U\to V$ where
  $U$ is a subquotient of~$\Std_n^{\otimes k}$ in~$\mcC_{|E|^n}$. We
  have $\omega(e)=|E|^n-|V|$, and hence there is equality if $n\geq k$ (note
  that in $\mcC_{|E|^n}$, the tensor product is defined using the
  \emph{direct sum} of finite-dimensional $E$-vector spaces).

  On the one hand, for all $n\geq 1$, the function
  $g\mapsto |\Fix(g)|$ is the character of the standard representation,
  and on the other hand, by Knop's construction, the dimension
  \[
  \dim_{\Cc(t)} \Hom_{\mcC_t}(1_t,\Std_t^{\otimes k})
  \]
  is the number of subspaces of $E^k$. Thus, $\expect(|\Fix(X_n)|^k)$
  converges to this number. To conclude, we need however to apply
  Lemma~\ref{lm-hb} below, since in this case the size of the moments do
  not satisfy the Carleman condition, but it is known that
  \[
  |\{\text{subspaces of $E^k$}\}|\ll |E|^{k(k+1)/4}.\qedhere
 \]
\end{proof}

\begin{lemma}[Heath--Brown]\label{lm-hb}
  Let~$q\geq 1$ be an integer, and let $(m_k)_{k\geq 0}$ be a sequence
  of real numbers such that $m_k\ll q^{k(k+1)/4}$ for $k\geq 0$.
  Let $(Z_n)_{n \geq 1}$ be a sequence of random variables such that
  \begin{enumerate}
  \item[\emph{(1)}] For all $n$, the support of~$Z_n$ is contained in the set of
  powers $q^r$ for $r\geq 0$.
  \item[\emph{(2)}] For all $k\geq 0$, we have $\expect(Z_n^k)\to m_k$.
  \end{enumerate}
  Then $(Z_n)$ converges in law to a random variable $Z$ supported on
  powers of~$q$ with moments~$m_k$ for all $k\geq 0$.
\end{lemma}

\begin{proof}
  This is implicit in~\cite[Lemmas\,17 and 18]{heath-brown}.  More
  precisely, it follows from standard results in the method of moments
  that the second assumption implies that any subsequence of $(Z_n)_{n \geq 0}$
  which converges in law has a limit with moments~$m_k$, and it is
  elementary from the first assumption that all such limits are
  supported on powers of~$q$.  Heath--Brown's result (proved
  in~\cite{heath-brown} in the case $q=4$, but with immediate
  generalization) is that there is a unique probability measure on~$\Rr$
  with these two properties. Since moreover the convergence of moments
  implies uniform integrability (or tightness), this means that the
  sequence $(Z_n)_{n \geq 0}$ is relatively compact and has a unique limit point, and 
  hence converges. The stated properties of the limit are then clear.
\end{proof}

\begin{remark}
  A result of Christiansen~\cite{christiansen} (also cited by Fulman and
  Stanton) shows that the limiting measure of
  Proposition~\ref{pr-fulman}, as a measure on~$\Rr$, is not
  characterized by its moments. Thus, some extra condition is necessary to ensure
  uniqueness, and this is provided by the assumption that the support is
  restricted to powers of~$q$.
\end{remark}

Considering another example of Knop leads by the same method to a
similar result which is new, to the best of our knowledge.

\begin{proposition}\label{pr-bis}
  Let~$E$ be a finite field and let $(X_n)_{n \geq 1}$ be a sequence of random
  variables with~$X_n$ uniformly distributed in the affine-linear group
  $\Aff_n(E)$ of~$E^n$.

  The sequence $(|\Fix(X_n)|)_{n\geq 1}$, where $\Fix(g)$ is the set of
  fixed points of~$g\in\Aff_n(E)$, converges in law as $n\to
  +\infty$. For $k\geq 0$, the $k$-th moment of the limiting
  distribution is equal to the number of affine subspaces of $E^{k-1}$.
  \par
  Moreover, the $k$-th moment of $|\Fix(X_n)|$ is equal to the limiting
  moment for $n\geq k$.
\end{proposition}

\begin{proof}
  We argue as above with the base category $\mathcal{A}$ of (non-empty)
  affine spaces over~$E$
  (see~\cite[p.\,597,\,Ex.\,6;\,p.\,607,\,Ex.\,7]{knop}).
\end{proof}

\subsection{Fixed-point statistics for complex vector spaces}\label{sec-3.3}

  It is also natural to consider the category $\Rep(GL_t)$ of Deligne and Milne (see~\cite[\S\,10,\,Déf.\,10.2]{deligne}), interpolating the categories of representations of $\GL_n(\Cc)$. Indeed, the argument applies
  rather similarly, and leads to the analogue of Theorem~\ref{th-fix} in
  this context: the direct image under the
  trace $\mathrm{Tr}\colon U_n \to \Cc$ of the probability Haar measure on the unitary
  group~$U_n$ converges as~$n\to+\infty$ to a standard complex
  gaussian. This was first proved by Diaconis and Shashahani~\cite{d-s};
  see also Larsen's paper~\cite{larsen} for the case of the symplectic
  or orthogonal groups
  and real gaussians.  

  First, by Corollary~\ref{cor-general-invariant}, the assignment
  \[
  i(M)=\dim_{\Cc(t)}\Hom_{\Rep(GL_t)}(1_t, M)
  \] defines a positive invariant on
  $\Rep(GL_t)$. One can then show that there exists an object~$\Std_t$, which for $t=n$ corresponds to the standard representation of $\GL_n(\Cc)$ through the equivalence from the semisimplification of $\Rep(GL_t)$ to the category of representations of $\GL_n(\Cc)$, satisfying 
  \[
    i(\Std_t^{\otimes a}\otimes\dual(\Std_t)^{\otimes b})=\dim_{\Cc(t)}
    \Hom(1_{t}, \Std_t^{\otimes a}\otimes\dual(\Std_t)^{\otimes
      b})=
    \begin{cases} 0 & \text{if $a\not=b$,} \\
      a! & \text{if $a=b$.}
    \end{cases}
  \]
More precisely, with the notation of~\cite[Déf.\,10.2]{deligne}, the
  object $\Std_t$ corresponds to the pair of finite sets $(\{1\},\emptyset)$ and is denoted by $X_0^{\otimes \{1\}}$. Thus, 
  $\Std_t^{\otimes a}\otimes \dual(\Std_t)^{\otimes b}$ corresponds to the
  pair $(\{1,\ldots, a\},\{1,\ldots,b\})$ and the value of
  $i(\Std_t^{\otimes a}\otimes \dual(\Std_t)^{\otimes b})$ is the
  dimension of the space
  \[
  \Hom((\emptyset,\emptyset),(\{1,\ldots,a\},\{1,\ldots, b\})), 
  \]
  which is by definition the number of bijections
  $\{1,\ldots, b\}\to \{1,\ldots, a\}$. These values are known to be equal to the moments
  \[
  \frac{1}{\pi} \int_{\Cc}z^a\bar{z}^b e^{-|z|^2}dz
  \]
  of a standard complex gaussian random
  variable, which is therefore the spectral measure associated
  to~$\mathrm{Std}_t$. Using a stabilization property of the
  corresponding invariants for $\GL_n(\Cc)$ when $n>a+b$, one gets
  convergence as before (\cf~\cite[Prop.\,10.6]{deligne}).

 This proof is not as satisfactory as that of
  Theorem~\ref{th-fix}, because Deligne and Milne's definition of $\Rep(GL_t)$
  involves some \emph{a priori} knowledge of stability properties of
  representations and linear invariants of $\GL_n(\Cc)$. The argument does show, however, that the convergence to
  the gaussian can be interpreted in terms of spectral measures, and
  that the standard gaussian can also be interpreted as a ``generalized''
  Sato--Tate measure.  Moreover, it suggests the question:
  what are the spectral measures for other objects of
  $\Rep(GL_t)$?

 \begin{remark}\label{non-uniqueness}
   Since Berg~\cite{berg} proved that the third power of a real
   gaussian random variable is not determined by its moments, the
   third tensor power of $\Std_t \oplus \dual(\Std_t)$ (which, thanks
   to Proposition~\ref{pr-image-measures}, has spectral measure the
   cube of a real gaussian), gives an example of an object of
   $\Rep(GL_t)$ whose spectral measure is not unique. Once a spectral
   measure is not unique, it is a classical fact from the solution of
   the Hamburger moment problem that the set of all possible
   $\mu(i, M)$ has a rather complicated structure, as explained in
   \cite[Ch.\,7]{schmudgen}.
 
   % (2) One can argue similarly with the category $\Rep(O(t))$ of
   % Deligne~\cite[\S\,9,\,Déf.\,9.2]{deligne}: the standard
   % object~$\Std_t$ in this category is self-dual and satisfies
   % \[
   %   \dim_{\Cc(t)} \Hom_{\Rep(O(t))}(1_t,\Std_t^{\otimes a})=
   %   |\{\text{partitions of $\{1,\ldots,a\}$ with all parts of
   %     size~$2$}\}|
   % \]
   % by definition, which coincides with the $a$-th moment of the
   % standard real gaussian (namely, it is~$0$ when~$a$ is odd, and
   % equal to $a!/(2^{a/2}(a/2)!)$ if~$a$ is even). Thus the standard
   % real gaussian is the (unique) spectral measure of~$\Std_t$
   % in~$\Rep(O(t))$ relative to the invariant $M\mapsto \dim_{\Cc(t)}
   % \Hom(1_t,M)$. 
 \end{remark}

 One can argue similarly with the category $\Rep(O(t))$ of
 Deligne~\cite[\S\,9,\,Déf.\,9.2]{deligne}, which interpolates representations of orthogonal
 groups: the standard object~$\Std_t$ in this category is self-dual and
 satisfies
 \[
   \dim_{\Cc(t)} \Hom_{\Rep(O(t))}(1_t,\Std_t^{\otimes a})=
   |\{\text{partitions of $\{1,\ldots,a\}$ with all parts of
     size~$2$}\}|
 \]
 by definition, which coincides with the $a$-th moment of the
 standard real gaussian (namely, it is~$0$ when~$a$ is odd, and
 equal to $a!/(2^{a/2}(a/2)!)$ when~$a$ is even). Thus, the standard
 real gaussian is the (unique) spectral measure of~$\Std_t$
 in~$\Rep(O(t))$ relative to the invariant $i(M)=\dim_{\Cc(t)}
 \Hom(1_t,M)$. The corresponding convergence theorem is that of the
direct image under the trace of the probability Haar measure of~$O(n)$ to
 the standard real gaussian.

\section{Proof of Theorem~\ref{th-fi}}

Let~$m\geq 1$ be an integer and let~$\lambda$ be a partition
of~$m$ with parts $\lambda_1 \geq \lambda_2 \geq \cdots$. Recall that the statement concerns the limiting behavior of the measures~$(\chi_{\lambda,n}(X_n))_{n \geq m+\lambda_1}$, where~$X_n$ is a uniformly distributed random permutation in $S_n$ and $\chi_{\lambda, n} \colon S_n \to \Cc$ is the character of the representation of $S_n$ corresponding to the partition $(n-m, \lambda_1, \lambda_2, \dots)$. 

The argument will consist of two stages:
\begin{itemize}
\item We prove \emph{a priori} that the sequence of
  measures~$(\chi_{\lambda,n}(X_n))_{n \geq m+\lambda_1}$ converges in law as $n\to+\infty$ to some measure $\mu_\lambda$. 
\item We compute the moments of the limiting measure and show that they
  coincide with those of a spectral measure of the object $x_{\lambda, t}$ of the Deligne--Knop category. 
\end{itemize}

%\begin{remark} In general, the moments of the limiting measure are too large to deduce the convergence in law directly from the second step. However, one may speculate that it might be possible to do so using the recent results of Sawin and Wood~\cite[Prop.\,6.24]{sawin-wood}.\end{remark}

%\needspace{10\baselineskip}
\begin{lemma}\label{eqn:lemmalimitcycle}
  The sequence~$(\chi_{\lambda,n}(X_n))_{n \geq m+\lambda_1}$ converges in law to a measure $\mu_\lambda$ as $n\to+\infty$.
\end{lemma}

\begin{proof}
  For each $i\geq 1$, let $\ell_i(\sigma)$ denote the number of $i$-cycles (fixed points if~$i=1$) in the representation of $\sigma$ as a product of cycles with disjoint support. It is
  known from the theory of symmetric functions that there exists a so-called character
  polynomial $q_{\lambda}\in \Qq[(L_i)_{i\geq 1}]$ such that, for all
  large enough $n$, the equality
  \[
  \chi_{\lambda,n}(\sigma)=q_{\lambda}(\ell_1(\sigma),\ldots,\ell_i(\sigma),
  \ldots)
 \]
  holds for all $\sigma \in S_n$ (see, for
  instance,~\cite[Ex.\,I.7.14]{macdonald}). Since the sequences $(\ell_i(X_n))_{i\geq 1}$ are also known to converge
  in law as $n\to +\infty$ to a sequence $(P_{1/i})_{i\geq 1}$ of independent Poisson random variables with parameters $1/i$
  (see, e.g.,~\cite[Th.\,7]{d-s}), the sequence~$(\chi_{\lambda,n}(X_n))_{n \geq m+\lambda_1}$ converges in law to $\mu_\lambda=q_{\lambda}(P_1,\ldots, P_{1/i},\ldots)$.
\end{proof}

\begin{remark} In the spirit of Remark \ref{remarkotherasympt}\,(2), it would be interesting to prove Lemma \ref{eqn:lemmalimitcycle} without using character polynomials. 

\end{remark}

The second step will rely on Deligne's construction of the category of representations of~$S_t$, which enjoys some functoriality properties that have not been explicitly established by Knop. Since Theorem~\ref{th-fi} is
new, the fact that Deligne's definition involves some a priori knowledge
of representations of the symmetric groups is not an instance of
circular reasoning.

Let $A$ be a commutative ring and $t \in A$. We use the
$A$-linear category $\Del(S_t,A)$ of Deligne~\cite[Déf.\,2.17]{deligne}, keeping the notation $\Del(S_t)$ for $A=\Cc(t)$ and the indeterminate~$t$. The basic objects of this category are associated to finite sets $U$ and denoted by\footnote{Deligne's basic generators~$[U]$ are not the same as the basic objects in Knop's definition, but the
  precise relation between them is explained by Knop
  in~\cite[Rem.\,1.2]{knop2}.} $[U]$; their Hom spaces are introduced in~\cite[Déf.\,2.12]{deligne}. For each integer~$N\geq 0$, we consider the full subcategory $\Del(S_t,A)^{(N)}$ whose objects are
the direct factors of sums of~$[U]$ for $U$ of cardinality~$\leq N$. Deligne~\cite[Prop.\,5.1]{deligne} proved that
$\Del(S_t)^{(N)}$ is a semisimple abelian category if $t$ is not an
integer between~$0$ and~$2N-2$. Moreover, under the assumptions 
\begin{equation}\label{eq-assumption}
  t-k\in A^{\times}\text{ for } 0\leq k\leq 2N-2 \quad \text{and}\quad 
  N!\in A^{\times}, 
\end{equation} he associated to
any pair $(y,\rho)$ consisting of a finite set~$y$ with $|y|\leq 2N$
and an irreducible representation~$\rho$ of the symmetric group~$S_y$,
an object $\uple{x}_{y,\rho,A}$ of $\Del(S_t,A)^{(N)}$; see~\hbox{\cite[Prop.\,5.1 and Rem.\,5.6]{deligne}.} (This object is independent, up to isomorphism, of the choice of~$N$, provided \eqref{eq-assumption} holds, and hence the value of~$N$ is omitted from the
notation.)

The objects $\uple{x}_{y,\rho,A}$ are functorial with respect to~$A$ under the natural
base-change functor
\[
T_{A,B}\colon \Del(S_t,A)\longrightarrow \Del(S_t,B)
\]
when~$B$ is an~$A$-algebra (see~\cite[Déf.\,2.17]{deligne}), i.e., there
are isomorphisms
\[
\uple{x}_{y,\rho,B}\simeq T_{A,B}(\uple{x}_{y,\rho,A}).
\]
If $B$ is a field of characteristic zero and the image of $t$ is not a non-negative integer,
then the full category $\Del(S_t,B)$ is a semisimple abelian category, and its simple objects are precisely those of the form $\uple{x}_{y,\rho,B}$, for a unique pair $(y,\rho)$, up to isomorphism.

% ; in particular, if
% furthermore~$\Cc(t)$ (resp.\,$\Cc$) is an $A$-algebra, the properties
% and functoriality of the construction (\loccit) imply that the object
% $\uple{x}_{\lambda,A}$ is mapped to simple objects $x_{\lambda,t}$ and
% $x_{\lambda,n}$ of~$\Del(S_t)$ and~$\Del(S_n)^{(N)}$, respectively,
% under the natural base-change functors
% $$
% \Del(S_t,A)\to \Del(S_t),\quad\quad \Del(S_t,A)\to \Del(S_n)
% $$

From now on, we fix an integer~$N \geq 1$ and consider the ring 
\[
A=\Cc[t]\Bigl[\Bigl(\frac{1}{t-k}\Bigr)_{0\leq k\leq 2N-2}\Bigr], 
\]
which is a principal ideal domain (being a localization of the principal ideal domain $\Cc[t]$) and satisfies the assumption~(\ref{eq-assumption}). 

Let $m\geq 1$ be an integer and~$\lambda$ a partition of~$m$. We then
set $\uple{x}_{\lambda,A}=\uple{x}_{y,\rho,A}$, where
$y=\{1,\ldots, m\}$ and $\rho$ is the irreducible representation
of~$S_m$ associated to the partition~$\lambda$. We denote by
$x_{\lambda,t}$ the base change of $\uple{x}_{\lambda,A}$ to
$\Del(S_t)$ under the natural inclusion $A \hookrightarrow \Cc(t)$. Furthermore, if $n>2N-2$, then we denote by
$x_{\lambda,n}$ the base change of $\uple{x}_{\lambda,A}$ to
$\Del(S_n)$ under the morphism $A \to \Cc(t)$ that maps~$t$ to $n$.

% If furtherfore~$\Cc(t)$ (resp.\,$\Cc$) is an $A$-algebra, the properties
% and functoriality of this construction (\loccit) imply that the object
% $\uple{x}_{\lambda,A}$ is mapped to simple objects $x_{\lambda,t}$ and
% $x_{\lambda,n}$ of~$\Del(S_t)$ and~$\Del(S_n)^{(N)}$, respectively,
% under the natural base-change functors
% $$
% \Del(S_t,A)\to \Del(S_t),\quad\quad \Del(S_t,A)\to \Del(S_n)
% $$
% (see~\cite[Déf.\,2.17]{deligne}).

We begin with a lemma generalizing the first step of the proof of
Lemma~\ref{lm-dk-specialization} (in the sense that it shows that
certain Hom spaces have the same dimension in all Deligne--Knop
categories $\Del(S_t)$, even when $t$ is a non-negative integer,
provided it is ``large enough'').

\begin{lemma}\label{lm-dk-2}
  Let~$\lambda$ be a partition of an integer $m\geq 1$ and let~$a\geq 0$ be an integer. For any integer~$n \geq 4am-1$, the following equality holds: 
  \[
  \dim_{\Cc(t)} \Hom_{\Del(S_t)}(1_t,x_{\lambda,t}^{\otimes a}) =\dim_{\Cc}
  \Hom_{\Del(S_n)}(1_n,x_{\lambda,n}^{\otimes a}). 
  \]
\end{lemma}

\begin{proof}
  Let $N\geq 1$ be an integer such that
  $N\geq 2am$. Then both~$\uple{x}_{\lambda,A}$ and 
  $\uple{x}_{\lambda,A}^{\otimes a}$ are objects
  of~$\Del(S_t,A)^{(N)}$ (this follows from the fact that the tensor
  product of two basic objects $[U]$ and~$[V]$ is a direct sum of
  objects~$[W]$ with~$|W|\leq |U|+|V|$; see~\cite[\S\,5.10]{deligne}).
  Consequently, by~\cite[Rem.\,5.6]{deligne} and the fact that $A$ is principal, there is a direct sum decomposition
  \begin{equation}\label{eqn:dsd}
  \uple{x}_{\lambda,A}^{\otimes a}\simeq \bigoplus_{|\mu|\leq N}
  v(\mu)\uple{x}_{\mu,A}
  \end{equation}
  for some non-negative integers~$v(\mu)$, where the sum is over
  partitions of integers~$\leq N$. Assume~$n>2N-2$.  Applying base-change to~$\Cc(t)$ and to~$\Cc$ as above
  $t\mapsto n$, we derive from~\eqref{eqn:dsd} direct sum decompositions
  \[
  x_{\lambda,t}^{\otimes a}\simeq \bigoplus_{|\mu|\leq N}v(\mu)x_{\mu,t}
  \quad\quad x_{\lambda,n}^{\otimes a}\simeq \bigoplus_{|\mu|\leq
    N}v(\mu)x_{\mu,n}.
  \]

  By~\cite[Rem.\,5.6]{deligne}, the objects $\uple{x}_{\mu,A}$ have the property that
  \[
  \Hom(\uple{x}_{\mu,A},\uple{x}_{\nu,A})=\begin{cases} 0 &\text{ if }
  \mu\not=\nu, \\ A &\text{ if }
  \mu=\nu.\end{cases}
  \]
  Since the unit objects of~$\Del(S_t)$ and~$\Del(S_n)$ are $x_{\mu,t}$ and $x_{\mu,n}$, respectively, for the partitition $\mu=(m)$ corresponding to
  the trivial representation of~$S_m$, we therefore deduce from these
  decompositions that the equalities
  \[
  \dim_{\Cc(t)} \Hom_{\Del(S_t)} (1_t,x_{\lambda,t}^{\otimes a})=v((m))= \dim_{\Cc}
  \Hom_{\Del(S_n)}(1_n,x_{\lambda,n}^{\otimes a})
  \]
  hold for all $n \geq 4am-1$, which concludes the proof.
\end{proof}

\begin{remark}
  A combinatorial formula for
  \[
  \dim_{\Cc(t)} \Hom_{\Del(S_t)}(1_t,x_{\lambda,t}^{\otimes a})
  \]
  has been obtained (in the generality of tensor envelopes) by Knop~\cite[Cor.\,5.4,\,Ex.\,5.6]{knop3}.
\end{remark}

\begin{proof}[End of the proof of Theorem~\ref{th-fi}]
Let~$a\geq 0$ be an integer. By Lemma~\ref{lm-dk-2}, the equalities
\[
i(x_{\lambda,t}^{\otimes
  a})=\dim_{\Cc(t)}\Hom_{\Del(S_t)}(1_t,x_{\lambda,t}^{\otimes a})
=\dim_{\Cc}\Hom_{\Del(S_n)}(1_n,x_{\lambda,n}^{\otimes a})
\] hold for all large enough integers~$n$. Besides, Deligne~\cite[Prop.\,6.4]{deligne} has shown that, provided~$n>2m$, the semisimplification functor
\[
\Del(S_n)\to \Del(S_n)/\mathcal{N}_n=\mathrm{Rep}(S_n)
\]
maps the object~$x_{\lambda,n}$ to the representation~$\pi_{\lambda,n}$ of~$S_n$ associated to the partition $\lambda^{(n)}$. 
Thus, we obtain the lower bound
$$
i(x_{\lambda,t}^{\otimes
  a})=\dim_{\Cc}\Hom_{\Del(S_n)}(1_n,x_{\lambda,n}^{\otimes a})\geq
\dim_{\Cc} \Hom_{\mathrm{Rep}(S_n)}(1_n,\pi_{\lambda,n}^{\otimes a}),
$$
with equality if and only if
$\mathcal{N}(1_n,x_{n,\lambda}^{\otimes a})=0$. For all large
enough (depending on~$a$ and $\lambda$) integers $n$, we have
$\mathcal{N}(1_n,x_{\lambda,n}^{\otimes a})=0$ (e.g., by Knop's
criterion),  and hence for such~$n$, we get
\[
\int_{\Rr}x^a\mu_{\lambda}(x)= i(x_{\lambda,t}^{\otimes a})=
\dim_{\Cc} \Hom_{\mathrm{Rep}(S_n)}(1_n,\pi_{\lambda,n}^{\otimes a}) =
\frac{1}{n!}\sum_{\sigma\in S_n} \chi_{\lambda,n}(\sigma)^a, 
\] as we wanted to show. This concludes the proof of Theorem~\ref{th-fi}.  
\end{proof}

\section{Arithmetic speculations}\label{sec-speculations}

The distribution of the number of fixed points of random permutations
in~$S_n$ for a given integer~$n\geq 1$ occurs naturally in number
theory as a limiting distribution for the number of zeros modulo a
prime number~$p$ of a fixed polynomial with integer coefficients~$f\in\Zz[T]$ of degree $n$ and Galois group~$S_n$. Indeed, let~$\rho_f(p)$ be this number. A special
case of Chebotarev's density theorem states in that case\footnote{\ For an arbitrary irreducible polynomial~$f \in \Zz[T]$, the corresponding limit would be the probability that a uniformly distributed random element of the Galois group of the
  splitting field of~$f$, viewed as a permutation of the complex roots of~$f$, has $r$ fixed points.} that the
limit formula
\[
  \lim_{x\to+\infty} \frac{1}{\pi(x)}\big|\{p\leq x \mid \rho_f(p)=r\}\big|
  =\frac{1}{n!}  |\{\sigma\in S_n \text{ with } 
|\Fix(\sigma)|=r\}|
\]
holds for all integers~$r\geq 0$, where~$\pi(x)$ denotes the number of
primes~$p\leq x$ (this was already observed by
Kronecker~\cite{kronecker} in 1880, who also pointed out the limiting
behavior as~$n\to+\infty$).

One may ask if a similar framework can give rise to the Poisson
distribution, viewed as the number of fixed points of a ``random
element'' of~$S_t$ for an indeterminate~$t$. Some work of Kowalski and Soundararajan~\cite[\S\,2.4]{ks1} involving
\emph{pseudopolynomials} might be related. Indeed, they have
formulated the following conjecture:

\begin{conjecture}[Kowalski--Soundararajan]\label{conjKS}
  Let~$F(n)=\sum_{k=0}^n n!/k!$ for integers $n\geq 0$. For any
  prime number~$p$, let~$\rho_F(p)$ be the number of integers~$x$ satisfying 
  $0\leq x\leq p-1$ and~$F(x)\equiv 0\mods{p}$. Then, for each integer~$r\geq 0$, the following limit formula holds: 
  \[
    \lim_{x\to+\infty} \frac{1}{\pi(x)}\big|\{p\leq x \mid
    \rho_F(p)=r\}\big|= \frac{1}{e}\frac{1}{r!}. 
  \]
  \end{conjecture}

A \emph{pseudopolynomial} in the sense of Hall \cite{Hall} is a sequence $(a_n)_{n \geq 0}$ of integers such that \hbox{$m-n$} divides $a_m-a_n$ for all $m>n$. Setting $G(n)=a_n$, this condition guarantees that the value $G(x)\mods{p}$ is well-defined for~$x\in\Zz/p\Zz$, independently of the choice of a representative to compute it. Besides the sequences $(f(n))_n$ of values of a polynomial with integer coefficients $f \in \Zz[X]$, a standard example is~$F(n)$ as in Conjecture \ref{conjKS}. This function can also be written as $e\int_1^\infty x^{n} e^{-x}dx$ for all $n \geq 0$ (an incomplete gamma function), or $\lfloor en!\rfloor$ for~$n \geq 1$.  

Numerical evidence in favour of Conjecture \ref{conjKS} is quite
convincing~\cite[\S\,2.4]{ks1}. We speculate that, if true, this limiting
behaviour might be explained by appealing to the properties of~$S_t$
and some avatar of Chebotarev's density theorem.

Another tantalizing experimental parallel observation is the
following. It results from Deligne's equidistribution theorem and the
work of Katz (see~\cite[Th.\,7.10.6]{esde}) that, given a
polynomial~$f\in\Zz[X]$ of degree~$n\geq 6$ whose derivative
$f'$ has Galois group~$S_{n-1}$, the exponential sums
\[
  W_f(a;p)=\frac{1}{\sqrt{p}}
  \sum_{x\mods{p}} \exp\Bigl(2\pi i \frac{af(x)}{p}\Bigr)
\]
for $a\in (\Zz/p\Zz)^{\times}$ become equidistributed as~$p\to+\infty$
like the traces of random matrices in a compact group~$K\subset
\mathrm{U}_n$ which contains~$\mathrm{SU}_n$. 

By analogy and comparison with the results of
Diaconis--Shashahani and Larsen, we are then led to expect the following:
\begin{conjecture}
  Let~$F(n)=\sum_{k=0}^n n!/k!$ for integers $n\geq 0$. For a prime number~$p$ and $a\in (\Zz/p\Zz)^{\times}$, set
  \[
    W_F(a;p)=\frac{1}{\sqrt{p}}
    \sum_{x\mods{p}}\exp\Bigl(2\pi i \frac{aF(x)}{p}\Bigr).
  \] Then the values~$(W_F(a;p))_{a\in(\Zz/p\Zz)^{\times}}$ become
  equidistributed as~$p\to+\infty$ like a standard complex gaussian,
  i.e., for any continuous bounded function~$\varphi\colon \Cc\to\Cc$, the following holds: 
  \[
    \lim_{p\to+\infty} \frac{1}{p-1} \sum_{a\in(\Zz/p\Zz)^{\times}}
    \varphi(W_F(a;p))= \frac{1}{\pi} \int_{\Cc}\varphi(z)e^{-|z|^2}dz. 
  \]
\end{conjecture}

Numerical evidence is again very convincing here. A potential link suggests itself with the category~$\Rep(GL_t)$, and even more tantalizing
is the suggestion of a form of Schur--Weyl duality relating the categories $\Rep(S_t)$ and $\Rep(GL_t)$. 

\appendix 

\section{Knop's construction of the category
  $\Rep(S_t)$}\label{appendix}
%% EK

In this section, we recall the steps of Knop's construction of
tensor envelopes, specialized to the case of the opposite of the category of finite sets which leads to Deligne's category of ``representations'' of $S_t$. 

Given sets $X$, $Y$ and $Z$ with maps $f\colon Y\to X$ and
$g\colon Y\to Z$, we define the gluing~\hbox{$X\dunion_Y Z$} as the quotient
of the disjoint union $X\dunion Z$ by the smallest equivalence relation that identifies~$f(y) \in X$ with $g(y) \in Z$ for all $y\in Y$.

Recall that a partition of a set~$X$ is a set of non-empty subsets of~$X$,
pairwise disjoint and with union equal to~$X$; we will identify partitions with equivalence relations on~$X$.

Given sets $X$, $Y$ and $Z$, and partitions $\alpha$ of $X\dunion Y$ and
$\beta$ of $Y\dunion Z$, one defines a partition~$\beta\ccirc\alpha$ of
$X\dunion Z$ as follows:
\begin{itemize}
\item the equivalence class of an element~$x\in X$ is the union of the
  $\alpha$-equivalence class of~$x$ and of the set of $z\in Z$ such that
  there exists $y\in Y$ which is $\alpha$-equivalent to~$x$ and~$\beta$\nobreakdash-equivalent to~$z$; 
\item the equivalence class of an element~$z\in Z$ is the union of the
  $\beta$-equivalence class of~$z$ and of the set of $x\in X$ such that
  there exists $y\in Y$ which is $\alpha$-equivalent to~$x$ and~$\beta$\nobreakdash-equivalent to~$z$.
\end{itemize}

Using the quotient maps
\[
Y\to (X\dunion Y)/\alpha,\quad\quad Y\to (Y\dunion Z)/\beta,
\]
we define the gluing
$(X\dunion Y)/\alpha\ \dunion_Y\ (Y\dunion Z)/\beta$ as above. There is an
injective map
\[
j\colon (X\dunion Z)/\beta\ccirc \alpha \to (X\dunion Y)/\alpha\
\dunion_Y\ (Y\dunion Z)/\beta, 
\]
and we define $\gamma(\alpha,\beta)$ as the cardinality of the
complement of the image of~$j$. Concretely, this is the number of equivalence
classes of elements of~$Y$ which are not $\alpha$-equivalent to an element of~$X$ neither $\beta$-equivalent to an element of~$Z$.

We fix a ring $k$ and an element $t$ of~$k$. The category 
\[
\mcC_t=\Rep(S_t)
\] is
constructed in three~steps. One first defines a $k$-linear category $\mcC^0_t$: its objects are
finite sets, and the morphism space $\Hom_{\mcC_t^0}(X,Y)$ is the free
$k$-module generated by partitions of the finite set $X\dunion Y$. The
composition maps are the $k$-bilinear maps given by
\begin{align*}
\Hom_{\mcC_t^0}(Y,Z)\times \Hom_{\mcC_t^0}(X,Y)&\longrightarrow  \Hom_{\mcC_t^0}(X,Z) \\
(\beta,\alpha)&\longmapsto \beta\circ\alpha=t^{\gamma(\alpha,\beta)}\beta\ccirc\alpha.
\end{align*} Associativity is not obvious, and relates to
basic properties of the function $\gamma$.

If $f\colon X\to Y$ is a map of finite sets, then there is an associated
morphism $Y\to X$ in~$\mcC^0_t$ given by the smallest equivalence
relation~$\alpha_f$ on $Y\dunion X$ that identifies $x \in X$ with $f(x) \in Y$ for all $x \in X$. This construction gives rise to a contravariant functor from the
category of finite sets to the category $\mcC^0_t$ (because it is
elementary that $\gamma(\beta,\alpha)=0$ whenever $\alpha$ and $\beta$
are equivalence relations associated to maps, and hence $\alpha_{g \circ f}=\alpha_g \circ \alpha_f$ holds for composable maps $f$ and $g$); this functor is faithful.

From $\mcC^0_t$, a category $\mcC'_t$ is constructed as the category of
formal finite direct sums of objects of~$\mcC^0_t$, with morphisms given
by matrices in the obvious way. Finally, Knop's tensor envelope category
$\mcC_t$ is defined by ``adding images of projectors'': an object is a
pair $(X,p)$ of an object $X$ of $\mcC'_t$ and an endomorphism $p$ of~$X$
such that $p\circ p=p$, and
$$
\Hom_{\mcC_t}((X,p),(Y,q))=q \circ \Hom_{\mcC'_t}(X,Y)\circ p\subset
\Hom_{\mcC'_t}(X,Y).
$$

The category $\mcC^0_t$ admits a monoidal structure in the sense of \cite[Def.\,2.1.1]{etingof-al}. The tensor product bifunctor is
defined on objects as $X\otimes Y=X\dunion Y$ for finite sets $X$
and~$Y$. As for morphisms,  the tensor product
$ \alpha\otimes \beta\in \Hom_{\mcC^0_t}(X\otimes Y,X'\otimes Y')$ of $\alpha\in\Hom_{\mcC^0_t}(X,X')$ and
$\beta\in\Hom_{\mcC^0_t}(Y,Y')$ is
the equivalence relation on
\[
(X\otimes Y)\dunion (X'\otimes Y')= (X\dunion Y)\dunion (X'\dunion Y')
\]
which ``coincides'' with $\alpha$ on $X\dunion X'$ and with $\beta$ on
$Y\dunion Y'$. The commutativity constraint $X \otimes Y \stackrel{\sim}{\to} Y \otimes X$ and the associativity constraint $(X\otimes Y)\otimes Z\stackrel{\sim}{\to} X\otimes (Y\otimes Z)$ are given, respectively, by the morphisms associated to the obvious identifications $X \dunion Y=Y \dunion X$ and $(X\dunion Y)\dunion Z=X\dunion (Y\dunion Z)$. The unit object~$\mathbf{1}$ is the empty set, along with the unique morphism
$\mathbf{1}\otimes\mathbf{1}\to \mathbf{1}$.

It is then elementary that if $p$ and $q$ are projectors, then
$p\otimes q$ is also one, so the rules
\[
(X,p)\otimes (Y,q)=(X\otimes Y,p\otimes q)
\]
and bilinearity define a symmetric monoidal structure on~$\mcC_t$. 

The monoidal category $\mcC^0_t$ is rigid
(\cite[Def.\,2.10.1]{etingof-al}). Indeed, the dual of a finite set~$X$
is defined to be~$\dual(X)=X$ itself, and the evaluation and
coevaluation morphisms
$$
\mathrm{ev}_X\colon \dual(X)\otimes X\to \mathbf{1},\quad\quad
\mathrm{coev}_X\colon \mathbf{1}\to X\otimes \dual(X)
$$
are both identified with the equivalence relation on~$X\dunion X$
associated to the identity map on~$X$. For a
morphism~$\alpha\in\Hom_{\mcC^0_t}(X,Y)$, the transpose
$\tr{\alpha}\in\Hom_{\mcC^0_t}(Y,X)$ is defined as the composition
\begin{multline*}
  \dual(Y)=Y=Y\otimes \mathbf{1} \xrightarrow{\mathrm{id}\otimes
    \mathrm{coev}_X} Y\otimes (X\otimes X) \simeq (Y\otimes X)\otimes X
  \\
  \xrightarrow{(\mathrm{id}\otimes \alpha)\otimes \mathrm{id}}
  (\dual(Y)\otimes Y)\otimes X\xrightarrow{\mathrm{ev}_Y\otimes \mathrm{id}}
  \mathbf{1}\otimes X=X=\dual(X),
\end{multline*}
and corresponds to the obvious equivalence relation on~$X\dunion Y$
which is ``the same'' as $\alpha$ on~$Y\dunion X$.

The duality functor thus defined extends by linearity to $\mcC'_t$, and
finally to~$\mcC_t$: we have
\[
\dual(X,p)=(\dual(X),\mathrm{Id}_{\dual(X)}-\tr{p})
\]
for an object $(X,p)$ of $\mcC_t$, and
$\tr{(q\circ \alpha\circ p)}=\tr{p}\circ \tr{\alpha}\circ \tr{q}$ for
$q\circ \alpha\circ p\in \Hom_{\mcC_t}((X,p),(Y,q))$.

Thus, $\mcC_t$ has the structure of a rigid symmetric monoidal $k$-linear
category.

Suppose that the ring~$k$ is a field of characteristic~$0$
and~$t$ is not a non-negative integer. Then Knop proved that the category~$\mcC_t$ is semisimple~\cite[Th.\,6.1 along with Ex.\,1, p.\,596]{knop}. 

Let~$A$ be a fixed finite set and~$G=\Aut(A)$ the corresponding
symmetric group. An element $g \in G$ acts by precomposition with $g^{-1}$ on the sets of maps from $A$ to any finite set~$X$. The contravariant functor
\[
h_A(X)=\Hom_{\mathsf{Set}}(A,X)
\]
from the category of finite sets to the category
$\mathsf{Set}_G$ of finite sets with a $G$-action can be extended to a
tensor functor~$T_A\colon \mcC_t\to \mathsf{Rep}_k(G)$ of
finite-dimensional $k$\nobreakdash-linear representations of~$G$ so that the diagram
\[
\begin{tikzcd}
  \mathsf{Set}^{opp} \arrow[r,"h_A"] \arrow[d]&\mathsf{Set}_G\arrow[d]
  \\
  \mcC_t\arrow[r,"T_A"] & \mathsf{Rep}_k(G)
\end{tikzcd}
\]
commutes~\cite[proof of Th.\,9.4,\,(9.23)]{knop}, where the functor
$\mathsf{Set}_G\to \mathsf{Rep}_k(G)$ associates to a finite set~$Y$ with a
$G$-action the permutation representation of~$G$ on the free $k$-module with basis~$Y$.

\bibliographystyle{refs}
\bibliography{rencontres}

\end{document}